\documentclass[11pt,a4paper]{article}
\usepackage[a4paper]{geometry}
\usepackage{amsthm,amsfonts,amssymb,pifont,amsmath,mathrsfs,bm,eurosym}
\usepackage{type1cm,type1ec}
\usepackage[colorlinks,linkcolor=blue,citecolor=blue]{hyperref}
\usepackage{epsfig,color,graphicx,epsf}
\usepackage{makeidx}
\numberwithin{equation}{section} %公式编号
\renewcommand{\mathbf}[1]{\boldsymbol{#1}}  %重新定义 \mathbold 为矢量黑斜体
\newcommand{\Real}{\mathbb{R}}

\newcommand{\Sp}{\mathrm{Span\,}}

\newtheorem{Theorem}{\textsc{Theorem}}[section]
\newtheorem{Lemma}[Theorem]{\textsc{Lemma}}

\newtheorem{Definition}[Theorem]{\textsc{Definition}}

\newtheorem{Example}[Theorem]{\textsc{Example}}
\begin{document}
\title{\textbf{Polyhedron under Linear Transformations}}
\author{Zhang~Zaikun\footnote{Institute of Computational Mathematics and
                  Scientific/Engineering Computing,~Chinese Academy of Sciences,~Beijing 100190,~CHINA.}}
\date{May 6, 2008}
\maketitle
\begin{abstract}
The image and the inverse image of a polyhedron under a linear transformation are polyhedrons.

\noindent \textbf{Keywords:~}polyhedron,~linear transformation,~Sard quotient theorem.
\end{abstract}

\section{Introduction}
All the linear spaces discussed here are real.
\begin{Definition}~
\\\noindent i.)~Suppose that $X$ is a linear space,~a subset $P$~of $X$~is said to be a polyhedron if it has the form
$$P=\{x\in X;~f_k(x)\leq\lambda_i\},$$
where~$n$~is a positive integer,~$\{f_k\}_{k=1}^{n}\subset X'$,~and
$\{\lambda_k\}_{k=1}^{n}\subset \Real$.\\ If
$\lambda_k=0~(k=1,~2,~3,~...,~n)$,~then $P$~is said to be a
polyhedral cone.\vskip3pt
\noindent ii.)~Suppose that $X$ is a
TVS,~a subset $P$~of $X$~is said to be a closed polyhedron if it has
the form
$$P=\{x\in X;~f_k(x)\leq\lambda_i\},$$
where~$n$~is a positive integer,~$\{f_k\}_{k=1}^{n}\subset X^*$,~and
$\{\lambda_k\}_{k=1}^{n}\subset \Real$.\\ If
$\lambda_k=0~(k=1,~2,~3,~...,~n)$,~then $P$~is said to be a closed
polyhedral cone.
\end{Definition}
\noindent It is obvious that both $\emptyset$ and $X$ itself
are~(closed)~polyhedral cones.

\section{Main Results}
Our main results are as follows.
\begin{Theorem}\label{Linear}
Suppose that $X$ and $Y$ are linear spaces,~and $T: X\rightarrow Y$
is a linear operator. \\i.)~If $A\subset X$ is a
polyhedron~(polyhedral cone)~and $T$ is \textbf{surjective},~then
$T(A)$ is a polyhedron~(polyhedral cone).
\\ii.)~If $B\subset Y$ is a
polyhedron~(polyhedral cone),~then $T^{-1}(B)$ is a
polyhedron~(polyhedral cone).
\end{Theorem}
\begin{Theorem}\label{Frechet}
Suppose that $X$ and $Y$ are Fr$\acute{e}$chet spaces,~and $T:
X\rightarrow Y$ is a bounded linear operator. \\i.)~If $A\subset X$
is a closed polyhedron~(closed polyhedral cone)~and $T$ is
\textbf{surjective},~then $T(A)$ is a closed polyhedron~(closed
polyhedral cone).
\\ii.)~If $B\subset Y$ is a closed
polyhedron~(closed polyhedral cone),~then $T^{-1}(B)$ is a
polyhedron~(closed polyhedral cone).
\end{Theorem}
\noindent The conclusions above will be verified in section
\ref{Proofs of Main Results}.

\section{A Lemma}
The following conclusion is significant in our proof.
\begin{Lemma}[\textnormal{Sard Quotient Theorem}]~
\\i.)~Suppose that $X$,~$Y$ and $Z$ are linear spaces,~and
$S:X\rightarrow Y$,~$T:X\rightarrow Z$ are linear operators with
$S$~surjective.~If $\ker S\subset\ker T$,~then there exists a
uniquely specified linear operator $R:Y\rightarrow Z$,~such that
$T=RS$.
\\ii.)~Suppose that $X$,~$Y$ and $Z$ are TVS',~and
$S:X\rightarrow Y$,~$T:X\rightarrow Z$ are bounded linear operators
with $S$~surjective.~If $X$ and $Y$ are $Fr\acute{e}chet$ spaces and
$\ker S\subset\ker T$,~then there exists a uniquely specified
bounded linear operator $R:Y\rightarrow Z$,~such that $T=RS$.
\end{Lemma}
\begin{proof}
  We will prove only ii.).
  \\ Define
  \begin{align*}
  \tilde S:& X/\ker S\rightarrow Y\\
           & [x]\mapsto Sx,
  \end{align*}
  and
 \begin{align*}
  \tilde T:& X/\ker S\rightarrow Z\\
           & [x]\mapsto Tx.
  \end{align*}
  Then both $\tilde S$ and $\tilde T$ are well defined~(note that $\ker S\subset\ker
  T$)~and bounded.~Besides,~$\tilde S$~is bijective and $\tilde
  S^{-1}$~is bounded,~since $X/\ker S$~and $Y$ are both
  Fr$\acute{\textnormal{e}}$chet
  spaces.~Now define
  $$R=\tilde T\tilde S^{-1},$$then it is easy to show that $R$
  satisfies the requirements.
  \\\noindent The uniqueness of $R$ is trivial.
\end{proof}
\section{Proofs of Main Results}\label{Proofs of Main
Results} We will prove only theorem \ref{Frechet},~because the proof
of theorem \ref{Linear} is similar.~Only the \emph{polyhedron} case
will be discussed.
\begin{proof}[Proof of Theorem \ref{Frechet}]~
   \\i.)~Suppose that $$A=\bigcap_{k=1}^{n}\{x\in X;~f_k(x)\leq\lambda_k\},$$
   where~$n$~is a positive integer,~$\{f_k\}_{k=1}^{n}\subset X^*$,~and
$\{\lambda_k\}_{k=1}^{n}\subset \Real$.~The proof will be presented
in four steps.
  \\\underline{\textsc{Step} 1}.~We will prove that the conclusion
  holds if $$\ker T\subset \bigcap_{k=1}^{n}\ker f_k.$$
  In this case,~for any $k\in\{1,~2,~3,~...,~n\}$,~we can choose a functional~$g_k\in
  Y^*$ such that~$f_k=g_kT$~(Sard quotient theorem).~It can be shown
  without difficulty that $$T(A)=\bigcap_{k=1}^{n}\{y\in Y;~g_k(y)\leq \lambda_k\}.$$
  \\\underline{\textsc{Step} 2}.~We will prove that the conclusion
  holds if $\dim(\ker T)=1$.~This is the most critical part of the
  proof.\\Suppose that $\xi$~is a point in~$\ker T\setminus\{0\}$.~Let
  \begin{align*}
    K_+&=\{k;~1\leq k\leq n~\textnormal{and}~f_k(\xi)> 0\},\\
    K_-&=\{k;~1\leq k\leq n~\textnormal{and}~f_k(\xi)< 0\},\\
    K_0&=\{k;~1\leq k\leq n~\textnormal{and}~f_k(\xi)= 0\}.
  \end{align*}
 For any $i\in K_+$~and~$j\in K_-$,~define
  $$h_{ij}=f_i-\frac{f_i(\xi)}{f_j(\xi)}f_j.$$
Then define
\begin{align*}
A_1=&\bigcap_{^{i\in K_+,}_{j\in K_-}}\{x\in X;~h_{ij}(x)\leq
  \lambda_i-\frac{f_i(\xi)}{f_j(\xi)}\lambda_j\},\\
A_2=&\bigcap_{k\in K_0}\{x\in X;~f_k(x)\leq \lambda_k\}.
\end{align*}
If $K_+=\emptyset$~or~$K_-=\emptyset$,~we take $A_1$~as
$X$.~Similarly,~if $K_0=\emptyset$,~we take $A_2$~as $X$.~We will
prove that $T(A)=T(A_1\cap A_2)$.~It
  suffices to show
  that $T(A_1\cap A_2)\subset T(A)$.
\begin{itemize}
  \item If $K_+=\emptyset=K_-$,~nothing
  needs considering.
  \item If~$K_+\neq\emptyset=K_-$,~fix a point $x\in A_1\cap A_2$,~define
  $$s=\min_{i\in{K_+}}\frac{\lambda_i-f_j(x)}{f_i(\xi)},$$then it is
  easy to show that $x+s\xi\in A$~and~$T(x+s\xi)=Tx$.~The case
  with~$K_-\neq\emptyset=K_+$~is similar.
  \item If $K_+\neq\emptyset\neq K_-$,~fix a point $x\in A_1\cap
  A_2$,~define
$$t=\max_{~j\in
  K_-}\frac{\lambda_j-f_j(x)}{f_j(\xi)}$$ and consider $x+t\xi$.~It is obvious that
  $$T(x+t\xi)=y$$and that $$f_j(x+t\xi)\leq\lambda_j,~\forall j\in
  K_-\cup K_0.$$Suppose
  $$t=\frac{\lambda_{j_0}-f_{j_0}(x)}{f_{j_0}(\xi)}~~(j_0\in
  K_-),$$
  then for any $i\in K_+$,
  \begin{align*}
    &f_i(x+t\xi)\\
    =&h_{ij_0}(x+t\xi)+\frac{f_i(\xi)}{f_{j_0}(\xi)}f_{j_0}(x+t\xi)\\
    \leq&\lambda_i-\frac{f_i(\xi)}{f_{j_0}(\xi)}\lambda_{j_0}+\frac{f_i(\xi)}{f_{j_0}(\xi)}\lambda_{j_0}\\
    =&\lambda_i.
  \end{align*}
  Thus $x+t\xi\in A$.
\end{itemize}
\noindent It has been shown that $T(A_1\cap A_2)\subset T(A)$,~and
consequently~$T(A_1\cap A_2)=T(A)$.~According to
\underline{\textsc{Step} 1},~the conclusion holds under the
assumption $\dim(\ker
  T)=1$.
  \vskip5pt\noindent \underline{\textsc{Step} 3}.~We will prove by induction that the conclusion
  holds if $\dim(\ker T)$~is finite.
  \\If $\dim(\ker T)=0$,~then $T$~is an isomorphism as well as a homeomorphism~(inverse mapping theorem),~thus nothing needs proving.~Now suppose that the conclusion
  holds when $\dim(\ker T)\leq n$~($n\geq 0$).~To prove the case with $\dim(\ker
  T)=n+1$,~choose a point $\eta$~in $\ker T\setminus\{0\}$,~find a functional $F\in X^*$~such
  that $F(\eta)=1$~(Hahn-Banach theorem),~and define
  \begin{align*}
    \hat T:&X\rightarrow Y\times \Real\\
      &x\mapsto (Tx,F(x)),\\
    \pi:&Y\times \Real\rightarrow Y\\
    &(y,\;\lambda)\mapsto y.
  \end{align*}
  Then we have
  \begin{itemize}
    \item $T=\pi\hat T$;
    \item $\dim(\ker\hat T)=n$;
    \item $\dim(\ker\pi)=1$;
    \item both $\hat T$ and $\pi$
  are surjective bounded linear operators.
  \end{itemize}
Thus by the induction hypothesis and the conclusion
of~\underline{\textsc{Step} 2},~$T(A)$~is a closed polyhedron.
  \vskip5pt\noindent\underline{\textsc{Step} 4}.~Now consider the general case.
  \\Let
   $$M=(\bigcap_{k=1}^{n}\ker f_k)\bigcap(\ker T),$$
   then $M$ is a closed linear subspace of $M$,~and therefore $X/M$~is a Fr$\acute{\textnormal{e}}$chet space.~Define
 \begin{align*}
      \tilde T:&X/M\rightarrow Y\\
      &[x]\mapsto Tx,\\
         \tilde f_k:&X/M\rightarrow \Real\\
      &[x]\mapsto f_k(x)
  \end{align*}
  where $k=1,~2,~3,~...,~n$.~Then $\tilde T$ and $\tilde f_k$ are well defined,~$T$~is a bounded linear operator from $X/M$~onto
  $Y$,~and $\{\tilde f_k\}_{k=1}^{n}\subset(X/M)^*$.~Besides,~we
  have$$(\bigcap_{k=1}^{n}\ker \tilde f_k)\bigcap(\ker \tilde
  T)=\{0\},$$which implies that$$\dim(\ker\tilde T)\leq n.$$Now let
   $$\tilde A=\bigcap_{k=1}^{n}\{[x]\in X/M;~\tilde f_k([x])\leq
   \lambda_k\},$$then$$T(A)=\tilde T(\tilde A).$$ From what has been proved,~it
   is easy to show that $T(A)$ is a closed polyhedron.~
   \\\indent Proof of part
   i.)~has been completed.
 \vskip6pt\noindent ii.)~This part is much easier.~Suppose that $$B=\bigcap_{k=1}^{m}\{y\in Y;~g_k(y)\leq\mu_k\},$$
   where~$m$~is a positive integer,~$\{g_k\}_{k=1}^{m}\subset Y^*$,~and
$\{\mu_k\}_{k=1}^{m}\subset \Real$.~One can show
 without difficulty that
 $$T^{-1}(B)=\bigcap_{k=1}^{m}\{x\in X;~g_k(Tx)\leq\mu_k\},$$which is a
 closed polyhedron in $X$.
\end{proof}

\section{Remarks}
For part i)~of theorem \ref{Frechet},~the completeness conditions
are essential.~This can be seen from the following examples.
\begin{Example}
  Suppose that $(Y,\;\|\cdot\|_Y)$ is an infinite dimensional
  Banach space,~and $f$ is an unbounded
  linear functional on it\footnote{For a locally bounded TVS $Y$,~there exist unbounded linear functionals on $Y$ provided $\dim Y=\infty$.~One of them can be constructed as follows:~Let~$U$ be a bounded neighborhood of 0,~and $\{e_k;~k\geq 1\}\subset U$~be a sequence of linearly independent elements in $Y$.~Let $M=\Sp\{e_k;~k\geq 1\}$,~and define $g:M\rightarrow \Real$,~$\sum\alpha_k e_k\mapsto\sum k\alpha_k$.~Then extend $g$ to
  $Y$.}.~Let $X$ has the same elements and linear structure as
  $Y$,~but the norm on $X$ is defined by
  $$\|x\|_X=\|x\|_Y+|f(x)|.$$It is clear that the identify mapping $I:X\rightarrow Y$ is linear,~bounded and bijective.~Now consider $\ker
  f$.~It
  is a closed polyhedral cone in $X$,~while its image under $I$ is
  not closed in $Y$.
\end{Example}
\begin{Example}
  Suppose that $X$ is $\ell^1$.~Let $Y$ has the same elements and linear structure as
  $X$,~but the norm on $Y$ is defined by $$\|(x_k)\|=\sup_{k\geq
  1}|x_k|.$$Then $f:(x_k)\mapsto\sum x_k$ is a bounded linear
  functional on $X$,~while it is unbounded on $Y$.~Now consider the identify mapping again.
\end{Example}
\noindent The preceding examples also imply that inverse mapping
theorem and Sard quotient theorem do not hold without completeness
conditions.
\end{document}